\documentclass[reqno, 12pt]{amsart}
\pdfoutput=1
\makeatletter
\let\origsection=\section \def\section{\@ifstar{\origsection*}{\mysection}}
\def\mysection{\@startsection{section}{1}\z@{.7\linespacing\@plus\linespacing}{.5\linespacing}{\normalfont\scshape\centering\S}}
\makeatother

\usepackage{amsmath,amssymb,amsthm}
\usepackage{mathrsfs}
\usepackage{mathabx}\changenotsign
\usepackage{dsfont}

\usepackage[dvipsnames]{xcolor}
\usepackage[backref]{hyperref}
\hypersetup{
    colorlinks,
    linkcolor={red!60!black},
    citecolor={green!60!black},
    urlcolor={blue!60!black}
}

\usepackage{graphicx}

\usepackage{caption}
\usepackage{subcaption}
\usepackage[labelformat=simple,labelfont={}]{subcaption}

\usepackage[open,openlevel=2,atend]{bookmark}

\usepackage[abbrev,msc-links,backrefs]{amsrefs}
\usepackage{doi}

\renewcommand{\PrintDOI}[1]{\doi{#1}}

\usepackage[T1]{fontenc}
\usepackage{lmodern}
\usepackage[babel]{microtype}
\usepackage[english]{babel}

\usepackage{geometry}
\numberwithin{equation}{section}
\numberwithin{figure}{section}

\usepackage{enumitem}
\def\rmlabel{\upshape({\itshape \roman*\,})}

\usepackage{pgf,tikz}

\makeatletter
\def\greek#1{\expandafter\@greek\csname c@#1\endcsname}
\def\Greek#1{\expandafter\@Greek\csname c@#1\endcsname}
\def\@greek#1{\ifcase#1
	\or $\alpha$%
	\or $\beta$%
	\or $\gamma$%
	\or $\delta$%
	\or $\epsilon$%
	\or $\zeta$%
	\or $\eta$%
	\or $\theta$%
	\or $\iota$%
	\or $\kappa$%
	\or $\lambda$%
	\or $\mu$%
	\or $\nu$%
	\or $\xi$%
	\or $o$%
	\or $\pi$%
	\or $\rho$%
	\or $\sigma$%
	\or $\tau$%
	\or $\upsilon$%
	\or $\phi$%
	\or $\chi$%
	\or $\psi$%
	\or $\omega$%
\fi}
\def\@Greek#1{\ifcase#1
	\or $\mathrm{A}$%
	\or $\mathrm{B}$%
	\or $\Gamma$%
	\or $\Delta$%
	\or $\mathrm{E}$%
	\or $\mathrm{Z}$%
	\or $\mathrm{H}$%
	\or $\Theta$%
	\or $\mathrm{I}$%
	\or $\mathrm{K}$%
	\or $\Lambda$%
	\or $\mathrm{M}$%
	\or $\mathrm{N}$%
	\or $\Xi$%
	\or $\mathrm{O}$%
	\or $\Pi$%
	\or $\mathrm{P}$%
	\or $\Sigma$%
	\or $\mathrm{T}$%
	\or $\mathrm{Y}$%
	\or $\Phi$%
	\or $\mathrm{X}$%
	\or $\Psi$%
	\or $\Omega$%
\fi}

\AddEnumerateCounter{\greek}{\@greek}{24}
\AddEnumerateCounter{\Greek}{\@Greek}{12}

\makeatother

\let\polishlcross=\l
\def\l{\ifmmode\ell\else\polishlcross\fi}

\def\paragraph#1{%
  \noindent\textbf{#1.}\enspace}

\let\emptyset=\varnothing
\let\setminus=\smallsetminus

\makeatletter
\def\moverlay{\mathpalette\mov@rlay}
\def\mov@rlay#1#2{\leavevmode\vtop{   \baselineskip\z@skip \lineskiplimit-\maxdimen
   \ialign{\hfil$\m@th#1##$\hfil\cr#2\crcr}}}
\newcommand{\charfusion}[3][\mathord]{
    #1{\ifx#1\mathop\vphantom{#2}\fi
        \mathpalette\mov@rlay{#2\cr#3}
      }
    \ifx#1\mathop\expandafter\displaylimits\fi}
\makeatother

\DeclareFontFamily{U}  {MnSymbolC}{}
\DeclareSymbolFont{MnSyC}         {U}  {MnSymbolC}{m}{n}
\DeclareFontShape{U}{MnSymbolC}{m}{n}{
    <-6>  MnSymbolC5
   <6-7>  MnSymbolC6
   <7-8>  MnSymbolC7
   <8-9>  MnSymbolC8
   <9-10> MnSymbolC9
  <10-12> MnSymbolC10
  <12->   MnSymbolC12}{}
\DeclareMathSymbol{\powerset}{\mathord}{MnSyC}{180}

\usepackage{tikz}
\usetikzlibrary{calc,decorations.pathmorphing}
\pgfdeclarelayer{background}
\pgfdeclarelayer{foreground}
\pgfdeclarelayer{front}
\pgfsetlayers{background,main,foreground,front}

\let\epsilon=\varepsilon
\let\eps=\epsilon
\let\rho=\varrho
\let\theta=\vartheta
\let\kappa=\varkappa

\let\E=\EE

\def\PP{{\mathds P}}
\let\Prob=\PP

\newcommand{\cA}{\mathcal{A}}

\newcommand{\cE}{\mathcal{E}}

\theoremstyle{plain}
\newtheorem{thm}{Theorem}[section]
\newtheorem{theorem}[thm]{Theorem}

\newtheorem{prop}[thm]{Proposition}

\newtheorem{fact}[thm]{Fact}

\newtheorem{lemma}[thm]{Lemma}

\theoremstyle{definition}

\newtheorem{prob}[thm]{Problem}

\usepackage{accents}

\let\phi=\varphi

\begin{document}

\title[Tight multiple twins in permutations]{Tight multiple twins in permutations}

\author{Andrzej Dudek}
\address{Department of Mathematics, Western Michigan University, Kalamazoo, MI, USA}
\email{\tt andrzej.dudek@wmich.edu}
\thanks{The first author was supported in part by Simons Foundation Grant \#522400.}

\author{Jaros\l aw Grytczuk}
\address{Faculty of Mathematics and Information Science, Warsaw University of Technology, Warsaw, Poland}
\email{j.grytczuk@mini.pw.edu.pl}
\thanks{The second author was supported in part by the Polish NSC grant 2015/17/B/ST1/02660.}

\author{Andrzej Ruci\'nski}
\address{Department of Discrete Mathematics, Adam Mickiewicz University, Pozna\'n, Poland}
\email{\tt rucinski@amu.edu.pl}
\thanks{The third author was supported in part by the Polish NSC grant 2018/29/B/ST1/00426}

\begin{abstract}
Two permutations are similar if they have the same length and the same relative order. A collection of $r\ge2$ disjoint, similar subsequences of
a permutation $\pi$ form $r$-twins in $\pi$. We study the longest guaranteed length of $r$-twins which are tight in the sense that either each twin alone forms a block or their union does. We address the same question with respect to a random permutation.
\end{abstract}

\maketitle


\setcounter{footnote}{1}

\section{Introduction}

By a \emph{permutation} we mean any finite sequence of distinct integers. We say that two permutations $(x_1,\dots,x_k)$ and $(y_1,\dots,y_k)$ are \emph{similar} if their entries preserve the same relative order, that is, $x_i<x_j$ if and only if $y_i<y_j$ for all pairs $\{i,j\}$ with $1\leqslant i<j\leqslant k$. For instance, $(2,1,3)$ is similar to $(5,4,8)$.

Large pairs of similar sub-permutations (called \emph{twins}) in a given, or random, permutation have recently attracted some attention (cf. \cite{Gawron},\cite{BukhR},\cite{DGR}). Here we are exclusively devoted to twins which appear in blocks.

A \emph{block} in a permutation $\pi$ is any subsequence of $\pi$ occupying a non-empty segment of consecutive positions. For instance, the permutation below contains 8 blocks of length~6. Some of them, like the highlighted one, enjoy a property which is of special interest to us: they  consists of two similar blocks of length $3$:
$$(12,6,7,\colorbox{Lavender}{2,1,3},\colorbox{green}{5,4,8},13,10,9,11).$$

How long such \emph{``order repetitions''} must occur in every long permutation? Clearly, every permutation of length at least $2$ contains trivial order repetitions of length $1$. Surprisingly, this is all what you get: as proved by Avgustinowich, Kitaev, Pyatkin, and Valyuzhenich in~\cite{AvgustinovichKPV}, there exist arbitrarily long permutations without order repetitions of any length greater than 1. This result is a permutation analog of the famous  theorem of Thue \cite{Thue} from 1906,  establishing the existence of arbitrarily long words over a $3$-letter alphabet avoiding word repetitions of any length, even 1 (see \cite{Lothaire}). Both results are constructive and provide simple recursive procedures for generating these objects. 

In \cite{DGR} we introduced a stronger avoidance property of permutations, defined as follows. A block in a permutation forms \emph{tight twins} of \emph{length $k$} if it consists of two similar disjoint \emph{subsequences} of length $k$ each. For instance, in a permutation below there are tight twins of length $3$, namely $(2,1,3)$ and $(5,4,8)$, that do not form a repetition:
$$(12,6,7,\colorbox{Lavender}{2},\colorbox{green}{5},\colorbox{green}{4},\colorbox{Lavender}{1},\colorbox{Lavender}{3},\colorbox{green}{8}13,10,9,11).$$
Note that the containment of tight twins, and, in particular, order repetitions, is not monotone in the sense that the absence of tight twins of length $k$ does not exclude the presence of tight twins longer than $k$. By using the probabilistic method, we proved in~\cite{DGR}, Thm. 3.6, that there exist arbitrarily long permutations without tight twins longer than~$12$. Most likely this constant is not optimal, but it cannot go all the way down to $1$, as every permutation of length 6 contains tight twins of length 2 (see \cite{DGR}, Prop. 3.7).

In this paper we study  generalizations of such problems to multiple twins. Let $r\geqslant 2$ be a fixed integer and $\pi$ be a permutation. We say that a block in $\pi$ consisting of $r$ similar disjoint subsequences, each of length $k$, forms \emph{tight $r$-twins} of \emph{length} $k$. For example, the following permutation contains tight $4$-twins of length $3$, namely $(2,1,3)$, $(5,4,8)$, $(15,7,17)$, $(12,9,16)$:
$$(14,18,\colorbox{Lavender}{2},\colorbox{green}{5},\colorbox{green}{4},\colorbox{Lavender}{1},\colorbox{cyan}{15},\colorbox{yellow}{12},
\colorbox{cyan}{7},\colorbox{cyan}{17},\colorbox{green}{8},\colorbox{yellow}{9},\colorbox{yellow}{16},\colorbox{Lavender}{3},6,10,11,13).$$
How long tight $r$-twins must occur in every permutation of length $n$? How long tight $r$-twins are contained, with high probability, in a \emph{random} permutation of the set $[n]=\{1,2,\dots,n\}$?

Let $tt^{(r)}(n)$ denote the largest integer $k$ such that every permutation of length $n$ contains tight $r$-twins of length $k$. Our result from \cite{DGR}, mentioned above, states that $tt^{(2)}(n)\leqslant 12$. Here we prove (Theorem \ref{Theorem Tight Twins}) that for every $r\geqslant 3$ we have
$$tt^{(r)}(n)\leqslant 15r.$$
 Note that this bound is independent of $n$. In contrast, we show (Theorem \ref{bttRandom}) that a random permutation of $[n]$ with high probability contains tight $r$-twins of length $\sim\frac{\log n}{(r-1)\log \log n}$.

We also consider a related function $f(r,k)$ defined as the least $n$ such that every permutation of length $n$ contains tight $r$-twins of length \emph{exactly} $k$. It is not hard to see that $f(2,2)=6$. We also determined, with a little help of computer, that $f(3,2)=12$. However, for every $r\geqslant 2$ and $k\geqslant 3$, we provide a construction of arbitrarily long permutations avoiding tight $r$-twins of length $k$. In other words, $f(r,k)=\infty$ for all pairs $(r,k)$ with $r\geqslant 2$ and $k\geqslant 3$ (Propositions \ref{infty} and \ref{rinfty}). For the remaining cases we only have a quadratic lower bound $f(r,2)\geqslant r(r+5)-12$, $r\ge3$ (Proposition \ref{ttr2}).

Another, more relaxed variant of multiple twins can be defined as follows. A family of $r$ pairwise disjoint and similar blocks in a permutation  $\pi$, each of length $k$, is called \emph{block $r$-twins} of \emph{length} $k$. For example, the following permutation contains block $4$-twins of length $3$ (highlighted):
$$(14,\colorbox{Lavender}{2,1,3},18,6,10,\colorbox{green}{5,4,8},\colorbox{cyan}{15,7,17},11,13,\colorbox{yellow}{12,9,16}).$$
How long block $r$-twins must occur in every permutation of length $n$? How long block $r$-twins occur with high probability in a \emph{random} permutation of $[n]$? We answer both these questions with asymptotic precision.

Let $bt^{(r)}(n)$ denote the largest integer $k$ such that every permutation of length $n$ contains block $r$-twins of length $k$. We prove (Theorem \ref{Theorem Block Twins}) that $$bt^{(r)}(n) = (1+o(1)) \frac{\log n}{\log \log n},$$
where the term $o(1)$ hides a dependence on $r$.
We also demonstrate (Theorem \ref{Theorem Block Twins Random}) that a random permutation of length $n$ with high probability contains block $r$-twins of length $\sim\frac{r\log n}{(r-1)\log \log n}$. Both these results were first proved for $r=2$ in \cite{DGR}.

One may also ask a reverse question: given $k$ and a permutation $\pi$, for how large $r$, are there block or tight $r$-twins of length $k$ in $\pi$?
We denote these parameters by $r_{bt}^{(k)}(\pi)$ and $r_{tt}^{(k)}(\pi)$, resp.
We show, in particular, that for $n$ even, with high probability, $r_{tt}^{(2)}(\Pi_n)=n/2$, that is, a random permutation $\Pi_n$,  contains tight $n/2$-twins of length 2 (Theorem \ref{r2Pi}).

Finally, let us return to the starting point and define \emph{block-tight $r$-twins}  as block $r$-twins which are at the same time tight $r$-twins. This means that $r$ similar blocks occur in a permutation consecutively,  with no gaps in-between, as in the following example:
$$(14,18,6,\colorbox{Lavender}{2,1,3},\colorbox{green}{5,4,8},\colorbox{cyan}{15,7,17},\colorbox{yellow}{12,9,16},10,11,13).$$
For $r=2$ this notion coincides with the order repetitions discussed at the beginning.

Let $btt^{(r)}(n)$ denote the largest integer $k$ such that every permutation of length $n$ contains block-tight $r$-twins of length $k$. The result of Avgustinovich et al.~\cite{AvgustinovichKPV} mentioned above implies that $btt^{(r)}(n)=1$ for all $r\geqslant 2$. Curiously, a random permutation with high probability contains block-tight $r$-twins of length $\sim\frac{\log n}{(r-1)\log \log n}$, which asymptotically agrees with the case of tight $r$-twins (Theorem \ref{bttRandom}).

In the forthcoming sections we give proofs of the above stated results: about  $bt^{(r)}(n)$ in Section \ref{Sbt}, and about $tt^{(r)}(n)$ and $f(r,k)$ in Section \ref{Stt}. Section \ref{third} is devoted to functions  $r_{bt}^{(k)}(\pi)$ and $r_{tt}^{(k)}(\pi)$,  while all results about  the length of block, tight, and block-tight $r$-twins in random permutations are proved in Section \ref{Sr}.
The next section contains a technical probabilistic lemma, while the last one presents some open problems.

\section{Independence of  occurrences of twins}\label{Prelim} In this section we  prove a technical result which will be used in several places of the paper. It is about the conditional probabilities of occurrences of $r$-twins in a random permutation.

Let $\Pi_n$ be a random permutation chosen uniformly from the set of all $n!$ permutations of~$[n]$. For  integers $r,k\ge2$ and a family of $r$ pairwise disjoint subsets $A_1,\dots,A_r$ of $[n]$, each of size $k$, let $\cE(A_1,\dots,A_r)$ be the event that there are $r$-twins in $\Pi_n$ on positions determined by the subsets $A_1,\dots,A_r$. Then,
\begin{equation}\label{1|k!}
\PP(\cE(A_1,\dots,A_r)=1)=\frac{\binom nk\binom{n-k}k\cdots\binom{n-(r-2)k}k\cdot (n-(r-1)k)!\cdot 1}{n!}=\frac{1}{k!^{r-1}}.
\end{equation}


\begin{lemma}\label{ABCD}
	For integers $r,t,k\ge2$ let $A^{(i)}_j$, $j=1,\dots,r$, $i=1,\dots,t$, be $k$-elements subsets of~$[n]$ such that for each $i=1,\dots,t$ all sets $A_1^{(i)},\dots,A_r^{(i)}$ are pairwise disjoint and, for some $1\le s\le r$,
	$$\bigcup_{j=1}^sA_j^{(1)}\cap\bigcup_{i=2}^t\bigcup_{j=1}^rA_j^{(i)}=\emptyset.$$
	  Then, setting $\cE^{(i)}:=\cE(A_1^{(i)},\dots,A_r^{(i)})$,
	$$\PP(\cE^{(1)}\cap\cdots\cap\cE^{(t)})\le\frac1{k!^s}\PP(\cE^{(2)}\cap\cdots\cap\cE^{(t)})\qquad\mbox{for}\quad1\le s\le r-2 ,$$
	while for $s\ge r-1$,  event $\cE^{(1)}$ is mutually independent of the family of events $\{\cE^{(2)},\dots,\cE^{(t)}\}$, that is,

	 $$\PP(\cE^{(1)}\cap\cdots\cap\cE^{(t)})=\frac1{k!^{r-1}}\PP(\cE^{(2)}\cap\cdots\cap\cE^{(t)})=\PP(\cE^{(1)})\PP(\cE^{(2)}\cap\cdots\cap\cE^{(t)}).$$
	
\end{lemma}

\proof  Let $N:=N(A_j^{(i)}: 1\le i\le t,1\le j\le r)$ be the number of permutations of an $(n-sk)$-element set $D$ on positions in $[n]\setminus\bigcup_{j=1}^sA_j^{(1)}$, that is bijections $f:D\to[n]\setminus\bigcup_{j=1}^sA_j^{(1)}$, such that there are $(r-s)$-twins on position sets $A^{(1)}_{s+1},\dots,A^{(1)}_{r}$, as well as,
$r$-twins on position sets $A^{(i)}_{1},\dots,A^{(i)}_{r}$ for all $i=2,\dots,t$. 
Observe that
$$|\cE^{(1)}\cap\cdots\cap\cE^{(t)}|=\frac{n!N}{k!^s(n-sk)!}\quad\mbox{and}\quad|\cE^{(2)}\cap\cdots\cap\cE^{(t)}|\ge(n)_{sk}N,$$
where the  equality follows from the fact that once  the values of $\Pi_n(i)$ are fixed on $[n]\setminus\bigcup_{j=1}^sA_j^{(1)}$, the rest of $\Pi_n$ is determined by assigning $k$-element subsets to each position set $A_i^{(1)}$, $j=1,\dots,s$, while the inequality is a result of dropping the part of definition of $N$  requesting that  there are $(r-s)$-twins on position sets $A^{(1)}_{s+1},\dots,A^{(1)}_{r}$. For $s\ge r-1$, there is nothing to drop, so we have equality there.
Hence,
$$\PP(\cE^{(1)}\cap\cdots\cap\cE^{(t)})=\frac{n!N}{k!^s(n-sk)!n!}\le\frac{|\cE^{(2)}\cap\cdots\cap\cE^{(t)}|}{k!^s(n)_{sk}(n-sk)!}=
\frac1{k!^s}\PP(\cE^{(2)}\cap\cdots\cap\cE^{(t)}),$$
where, again, for $s=r-1$, we have equality. \qed

\medskip

The first statement of Lemma \ref{ABCD} will only be used in the proof of Theorem \ref{Theorem Block Twins Random}. The second one will be applied in several proofs, whenever independence of occurrences of $r$-twins will be sought, for example, in both applications of the Local Lemma.

\section{Block twins}\label{Sbt}

  We say that a collection of disjoint, similar blocks $\{\sigma_1,\dots,\sigma_r\}$ in a permutation $\pi$ form \emph{block $r$-twins} in $\pi$. Let $bt^{(r)}(\pi)$ denote the longest length of block $r$-twins in $\pi$, that is,
$$bt^{(r)}(\pi)=\max\{\text{$|\sigma_1|:(\sigma_1,\dots,\sigma_r)$ form block $r$-twins in $\pi$}\}$$
and let $$bt^{(r)}(n)=\min\{bt^{(r)}(\pi):\text{$\pi$ is a permutation of $[n]$}\}.$$
Note that containment of block $r$-twins of length $k$ is monotone, that is, their absence in a permutation excludes both, block $r$-twins of length $k+1$ and block $r+1$-twins of length $k$.

The goal of this section is to pin-point $bt^{(r)}(n)$ asymptotically.
To this end will  need the standard Local Lemma.

For   events $\cE_{1},\ldots ,\cE_{n}$ in any probability space, \emph{a dependency graph} $D=([n],E)$ is any graph on vertex set $[n]$ such that for every vertex $i$ the event $\cE_i$ is jointly independent of all events $\cE_j$ with $ij\not\in E$.

\begin{lemma}[The Local Lemma; Symmetric Version \cite{ErdosLovasz} (see \cite{AlonSpencer})]\label{LLL Symmetric}
	Let $\cE_{1},\ldots ,\cE_{n}$ be events in any probability space. Suppose that the maximum degree of a dependency graph of these events is at most $\Delta$, and $\PP(A_i)\leqslant p$, for all $i=1,2,\dots,n$. If $ep(\Delta+1)\leqslant 1$, then $\PP \left(
	\bigcap\limits_{i=1}^{n}\overline{\cE_{i}}\right) >0$.
\end{lemma}

The following result gives an asymptotic formula for the function $bt^{(r)}(n)$. The term $o(1)$ depends on $r$.

\begin{thm}\label{Theorem Block Twins}
	We have
	\[
	bt^{(r)}(n) = (1+o(1)) \frac{\log n}{\log \log n}.
	\]
\end{thm}

\begin{proof}
	First we show the lower bound. Let $n = k((r-1)k! +1)$ and let $\pi$ be any permutation of~$[n]$. Divide $\pi$ into $(r-1)k!+1$ blocks, each of length $k$. By the pigeonhole principle, there are $r$ blocks that induce similar sub-permutations (forming thereby  $r$-twins). The choice of $n$, together with the Stirling formula, imply that $k = (1+o(1))\frac{\log n}{\log \log n}$.
	
	For the upper bound we use the probabilistic method based upon Lemma \ref{LLL Symmetric} and Lemma~\ref{ABCD}. Let $n = \left\lfloor k!(erk)^{-1/(r-1)}\right\rfloor$ and let $\Pi:=\Pi_n$ be a random permutation.
	An $r$-tuple of indices $i_1,\dots,i_r$ satisfying
	$$1\le i_1\le i_2-k\le i_3-2k\cdots \le i_r-(r-1)k\le n-rk,$$
	is called \emph{$k$-spread}. For a $k$-spread $r$-tuple define the event $\cE_{i_1,\dots,i_r}$ that segments $(\Pi(i_j),\Pi(i_j+1)),\dots,\Pi(i_j+k-1))$, $j=1,\dots,r$,  form block $r$-twins in $\Pi$.
We are going to apply Lemma \ref{LLL Symmetric} to events $\cE_{i_1,\dots,i_r}$ over all choices of $k$-spread $r$-tuples $i_1,\dots,i_r$.

By \eqref{1|k!}, we may set $p:=\PP(\cE_{i_1,\dots,i_r})=1/k!^{r-1}$.
	Notice that by Lemma \ref{ABCD}, case $s=r-1$, a fixed event $\cE_{i_1,\dots,i_r}$  is jointly independent of all the events  $\cE_{i_1',\dots,i'_r}$ for which
	$$\bigcup_{j=1}^{r-1}\{i_j,i_j+1,\dots,i_j+k-1\}\cap\bigcup_{j=1}^r\{i'_j,i_j'+1,\dots,i_j'+k-1\}=\emptyset.$$
	Thus, there is a dependency graph $D$ for these events with maximum degree at most
	\[
	\Delta = (r-1)kn^{r-1} \le rkn^{r-1}-1.
	\]
	This and the choice of $n$ yields that
	\[
	e(\Delta+1)p \le e \cdot rkn^{r-1} \cdot \frac{1}{k!^{r-1}} \le 1.
	\]
	Consequently,  Lemma \ref{LLL Symmetric} implies that there exists a permutation $\pi$ of $[n]$ with no block $r$-twins of length $k$, that is with $bt^{(r)}(\pi)<k$. In turn, $bt^{(r)}(n)\le bt^{(r)}(\pi)<k$. Again, the Stirling formula yields that $k = (1+o(1))\frac{\log n}{\log \log n}$.
\end{proof}

\section{Tight twins}\label{Stt}
In this section we consider  $r$-twins whose union occupies a block of consecutive positions in a permutation $\pi$. We call them \emph{tight $r$-twins}. Note that, unlike block twins, tight twins are not `monotone', that is the absence of tight $r$-twins of length $k$ in a permutation does not exclude the presence of longer tight $r$-twins. Likewise, it does not exclude the presence of tight $(r+1)$-twins of length $k$.

\subsection{Upper bound}

Let $tt^{(r)}(\pi)$ denote the maximum length of tight $r$-twins in $\pi$, that is,
$$tt^{(r)}(\pi)=\max\{\text{$|\sigma_1|:(\sigma_1,\dots,\sigma_r)$ is a collection of tight twins in $\pi$}\},$$
and let $$tt^{(r)}(n)=\min\{tt^{(r)}(\pi):\text{$\pi$ is a permutation of $[n]$}\}.$$

We will prove that for every fixed $r$ there is a constant $c=c(r)$ such that $tt^{(r)}(n)\leqslant c$ for all $n$.
We intend to apply again the probabilistic method.
However, due to the lack of monotonicity, in order to show that $tt^{(r)}(n)<k$, we need to find a permutation $\pi$ without $r$-twins of any length $m\ge k$. To this end, the most suitable tool seems to be  the following version of the Local Lemma, which is equivalent to the standard asymmetric version (see \cite{AlonSpencer}). The dependency graph was defined in Section \ref{Sbt}.

\begin{lemma}[The Local Lemma; Multiple Version (see \cite{AlonSpencer})]\label{LLL}
	Let $\cE_{1},\ldots ,\cE_{n}$ be events in any probability space with a dependency
	graph $D=(V,E)$. Let $V=V_{1}\cup \cdots \cup V_{t}$ be a partition such
	that all members of each part $V_{k}$ have the same probability $p_{k}$.
	Suppose that the maximum number of vertices from $V_{m}$ adjacent to a
	vertex from $V_{k}$ is at most $\Delta _{km}$. If there exist real numbers $%
	0\leq x_{1},\ldots ,x_{t}<1$ such that $p_{k}\leq
	x_{k}\prod\limits_{m=1}^{t}(1-x_{m})^{\Delta _{km}}$, then $\Pr \left(
	\bigcap\limits_{i=1}^{n}\overline{\cE_{i}}\right) >0$.
\end{lemma}

Equipped with this tool, we may now prove the main result of this section.

\begin{thm}\label{Theorem Tight Twins}
	For every $n\geqslant 1$ and $r\ge3$ we have $tt^{(r)}(n)\leqslant 15r$.
\end{thm}

\begin{proof}
	Let $\Pi$ be a random permutation of $[n]$. We will apply Lemma \ref{LLL} in the following setting. For a fixed block $K$ of length $rk$, $k\ge cr$ (where $c=c(r)$ will be specified later), let $\cA_K$ denote the event that a sub-permutation of $\Pi$ occupying $K$ consists of tight $r$-twins.  Let $V_k$ denote the collection of all such events $\cA_K$ for all possible blocks $K$ of length $rk$. Note that
$\cA_K=\bigcup_{K_1,\dots,K_r}\cE(K_1,\dots,K_r)$, where the union extends over all partitions of $K$ into $r$ disjoint subsets of size $k$ and $\cE(K_1,\dots,K_r)$ is the event defined prior to Lemma \ref{ABCD}.
Thus, by (\ref{1|k!}) and the union bound, for every $\cA_K\in V_k$,
	$$\PP(\cA_K)\le\sum_{K_1,\dots,K_r}\PP(\cE({K_1,\dots,K_r}))=\frac{1}{r!}\binom{rk}{k}\binom{(r-1)}{k}\cdots\binom{2k}{k}\cdot \frac{1}{(k!)^{r-1}}=\frac{1}{r!}\cdot \frac{(rk)!}{(k!)^{2r-1}}.$$ Hence, we may take $p_k=\frac{(rk)!}{r!(k!)^{2r-1}}$.
	
By Lemma \ref{ABCD}, case $s=r$, any event	$\cE({K_1,\dots,K_r})$ is mutually independent of all events $\cE({M_1,\dots,M_r})$ such that $\bigcup_{i=1}^r K_i\cap\bigcup_{i=1}^r M_i=\emptyset$. In turn,
Any event $\cA_K$ depends only on those events $\cA_M$ whose blocks $M$ intersect~$K$. Hence, if $M$ is any block of length $rm$, with $m\geqslant c$ and $M\neq K$, then we may take $\Delta_{km}=rk+rm-1$. Furthermore, we take $x_m=q^m$, where $q=q(r)\leqslant 1$ is a constant (to be specified later) such that $q^m\leqslant 1/2$ for $m\geqslant c$.
	
	We are going to prove that for every $k\geqslant c$,
	\[p_{k}\leq	x_{k}\prod\limits_{m=c}^{n/r}(1-x_{m})^{\Delta _{km}}.\]
	Since $x_m\leqslant1/2$ for $m\ge c$, we may use the inequality $1-x_m\geqslant e^{-2x_m}$ and obtain the bound
	
	\begin{align*}
	\prod\limits_{m=c}^{n/r}(1-x_{m})^{\Delta _{km}} \geqslant\prod\limits_{m=c}^{n/r}(1-x_m)^{r(k+m)}
	&\geqslant \exp\left( -2r\sum_{m=c}^{\infty}x_m(k+m)\right)\\
	&=\exp\left( -2rk\sum_{m=c}^{\infty}q^m\right) \cdot \exp\left(-2r\sum_{m=c}^{\infty}mq^m\right).
	\end{align*}
	Since $\sum_{m=c}^{\infty}q^m=\frac{q^c}{1-q}=:A$ and $\sum_{m=c}^{\infty}mq^m=\frac{q^c(-qc+q+c)}{(1-q)^2}=\frac{q^c}{1-q}\cdot\left (c+ \frac{q}{1-q}\right)= :B$, we will be done by showing that the following inequality holds for all $k\geqslant c$,
	$$\frac{(rk)!}{r!(k!)^{2r-1}}\leqslant \frac{q^k}{e^{2rkA}\cdot e^{2rB}}.$$
	It is not hard to see that when $r$ is fixed and $q=1/2$ (for instance), the inequality holds for sufficiently large $k$. Let us make more precise calculations to derive the dependence of $c$ on $r$.
	
	First we bound the left-hand side by using the well known bounds based on the Stirling formula, $n^ne^{-n} \sqrt{2\pi n} \le n! \le n^ne^{-n+1} \sqrt{n}$, which are valid for all positive integers $n$. Thus, we obtain
	\begin{equation*}
	\frac{(rk)!}{r!(k!)^{2r-1}}\leqslant \frac{(rk)^{rk}\cdot e \cdot \sqrt{rk}}{e^{rk}}\cdot\frac{e^{k(2r-1)}\sqrt{2\pi k}}{r!\cdot k^{k(2r-1)}\cdot(2\pi)^r\cdot k^r},
	\end{equation*}
	which simplifies to
	\begin{equation*}
	\frac{(rk)!}{r!(k!)^{2r-1}}\leqslant \frac{e\cdot\sqrt{2\pi r}}{r!\cdot (2\pi)^r}\cdot \frac{(r^r)^{k}\cdot e^{k(r-1)}}{k^{(k+1)(r-1)}}< (r^r)^k\cdot \left( \frac{e^k}{k^{(k+1)}}\right)^{r-1},
	\end{equation*}
	for $r\geqslant 2$. On the other hand, for $q=1/2$ we have $A=\frac{2}{2^c}$ and $B=\frac{2}{2^c}(c+1)$.  So, we will be done by showing that
	\begin{equation}\label{eq:ineq1}
	(r^r)^k\cdot (e^{r-1})^k\cdot (e^{4r/2^c})^k\cdot e^{{4r(c+1)/2^c}}\cdot2^k<k^{(k+1)(r-1)}
	\end{equation}
	holds for all $k\geqslant c$. Set $c=15r$.

	
	Since  $4r/2^c \le 1$ and ${4r(c+1)/2^c} \le 1$, we get
	\[
	(e^{4r/2^c})^k\le e^k \quad \text{ and } \quad e^{{4r(c+1)/2^c}}\le e.
	\]
	Thus, \eqref{eq:ineq1} will follow from
	\[
	(er)^{rk} e^{k+1} < k^{(k+1)(r-1)}.
	\]
	Finally, since $k\ge c = 15r \ge 2e^2 r$, we get
	\[
	k^{(k+1)(r-1)} \ge c^{(k+1)(r-1)} = (2r)^{(k+1)(r-1)} e^{2(k+1)(r-1)}
	\]
	and so it remains to show that
	$(2r)^{(k+1)(r-1)} \ge r^{rk}$ and $e^{2(k+1)(r-1)} \ge e^{rk+k+1}$. Observe that the first inequality is equivalent to $2^{(k+1)(r-1)} \ge r^{k-r+1}$, which holds since for $r\ge 3$,
	\[
	2^{(k+1)(r-1)} \ge 2^{kr/2} = \left(2^{r/2}\right)^k \ge r^k \ge r^{k-r+1}.
	\]
	The second inequality is equivalent to $k(r-3)+2r-3\ge 0$, which clearly holds, too.
	
\end{proof}
Note that the above defined $p_k$ is greater than $(r/ek)^{rk}r^{-r}\ge1$ for $2\le k\le r/(2e)$, $r$ large, so the  bound in Theorem \ref{Theorem Tight Twins} cannot be improved to $o(r)$ by this method (see Problems \ref{P61} and \ref{P62} in Section \ref{cr}).

\subsection{Function $f(r,k)$, or problem turned around}\label{around}

Here we put the cart before the horse and consider the following extremal problem.
Given integers $r,k\ge2$, determine the function
$$f(r,k)=\min\{n: \text{ every $\pi_n$ contains tight $r$-twins of length $k$}\}.$$
If no such $n$ exists, then we set $f(r,k)=\infty$. Note that $f$ is not monotone in either variable, that is, in general, it is not true that $f(r,k)\le f(r+1,k)$ or $f(r,k)\le f(r,k+1)$.
For example, permutation
$$\pi_{12}=(4, 5, 6, 9, 8, 7, 1, 2, 3, 12, 11, 10)$$
contains no tight 2-twins of length 3 (see the proof of Prop. \ref{infty} for explanation) but it does contain tight 3-twins of length 3, namely $(4,9,1)$, $(5, 8, 2)$, and $(6, 7, 3)$, all three similar to $(2,3,1)$. Also, $\pi_{12}$ does contain tight 2-twins of length 6, namely  $(4, 5, 6, 9, 8, 7)$ and $(1, 2, 3, 12, 11, 10)$, both similar to $(1,2,3,6,5,4)$. Owing to this inconvenience, there is no obvious relation between functions $f(r,k)$ and $tt^{(r)}(n)$.

As we show below in Propositions~\ref{infty} and \ref{rinfty}, quite surprisingly, $f(r,k)=\infty$ for any $r\ge2$ and $k\ge 3$. For  clarity of presentation we chose to first prove the case $r=2$, so that the general case will be easier to comprehend. For a sequence of integers $A=(a_1,\dots,a_m)$, set $-A=(-a_1,\dots,-a_m)$ and  $\overleftarrow{A}=(a_m,\dots,a_1)$ for the opposite and, resp., the inverse sequence to $A$.

\begin{prop}\label{infty} For all $k\ge3$, we have
	$f(2,k)=\infty$.
\end{prop}
\proof  For each $k\ge2$, we construct an infinite sequence of distinct integers which is free of tight 2-twins of length $2k-1$ and $2k$.
Consider a partition of all natural numbers into consecutive blocks of length $2k-1$,
$$\mathds{N}=\bigcup_{m\ge1} A_m,$$
where, for $m\ge1$, $A_m=((m-1)(2k-1)+1,\dots,m(2k-1))$ is viewed as a sequence.
Then we define
$$\pi^{(k)}=(-\overleftarrow{A_1})\overleftarrow{A_2}(-\overleftarrow{A_3})\overleftarrow{A_4}\cdots$$
For example,
$$\pi^{(2)}=(-3,-2,-1,\;6,5,4,\;-9,-8,-7,\;12,11,10,\cdots)$$
(see Figure~\ref{fig:pi2}).
Of course, for any fixed $n$ divisible by $2k-1$, we may extract a permutation of length $n$ as the initial segment and, to get rid of negative integers, rewrite it in the reduced form, that is, as a permutation of $[n]$ similar to it. For instance, for $k=2$ and $n=12$, we then recover the permutation $\pi_{12}$ presented above.

\begin{figure}

\begin{subfigure}[b]{0.4\textwidth}

\scalebox{0.5}
{

\begin{tikzpicture}
[line width = .5pt,
vtx/.style={circle,draw,black,very thick,fill=black, line width = 3pt, inner sep=2pt},
]
    \node[vtx] (v1) at (1,4) {};
    \node[vtx] (v2) at (2,5) {};
    \node[vtx] (v3) at (3,6) {};
    \node[vtx] (v4) at (4,9) {};
    \node[vtx] (v5) at (5,8) {};
    \node[vtx] (v6) at (6,7) {};
    \node[vtx] (v7) at (7,1) {};
    \node[vtx] (v8) at (8,2) {};
    \node[vtx] (v9) at (9,3) {};
    \node[vtx] (v10) at (10,12) {};
    \node[vtx] (v11) at (11,11) {};
    \node[vtx] (v12) at (12,10) {};

    \draw[line width=0.75mm, color=black]  (v1) -- (v2) -- (v3) -- (v4) -- (v5) -- (v6) -- (v7) -- (v8) -- (v9) -- (v10) -- (v11) -- (v12);

    \draw[dashed, line width=0.6mm, color=black] (v12) -- (13,5);

   \draw[color=gray] (0,6.5) -- (13,6.5);
   \draw[color=gray] (0,9.5) -- (13,9.5);
   \draw[color=gray] (0,3.5) -- (13,3.5);

   \fill[fill=black] (v1) circle (0.1) node [right] {\Large{$\; \pi(1)$}};
   \fill[fill=black] (v2) circle (0.1) node [right] {\Large{$\; \pi(2)$}};
   \fill[fill=black] (v3) circle (0.1) node [right] {\Large{$\; \pi(3)$}};
   \fill[fill=black] (v4) circle (0.1) node [right] {\Large{$\; \pi(4)$}};
   \fill[fill=black] (v5) circle (0.1) node [right] {\Large{$\; \pi(5)$}};
   \fill[fill=black] (v6) circle (0.1) node [right] {\Large{$\; \pi(6)$}};
   \fill[fill=black] (v7) circle (0.1) node [right] {\Large{$\; \pi(7)$}};
   \fill[fill=black] (v8) circle (0.1) node [right] {\Large{$\; \pi(8)$}};
   \fill[fill=black] (v9) circle (0.1) node [right] {\Large{$\; \pi(9)$}};
   \fill[fill=black] (v10) circle (0.1) node [right] {\Large{$\; \pi(10)$}};
   \fill[fill=black] (v11) circle (0.1) node [right] {\Large{$\; \pi(11)$}};
   \fill[fill=black] (v12) circle (0.1) node [right] {\Large{$\; \pi(12)$}};

   \draw[line width=1.25mm, color=black]  (v1) -- (v3);
   \draw[line width=1.25mm, color=black]  (v4) -- (v6);
   \draw[line width=1.25mm, color=black]  (v7) -- (v9);
   \draw[line width=1.25mm, color=black]  (v10) -- (v12);

   \fill[fill=black] (2,5.2) circle (0.0) node [left] {\Large{$-\overleftarrow{A_1}\ $}};
   \fill[fill=black] (5.1,7.55) circle (0.0) node [left] {\Large{$\overleftarrow{A_2}\ $}};
   \fill[fill=black] (8.1,2.6) circle (0.0) node [left] {\Large{$-\overleftarrow{A_3}\!$}};
   \fill[fill=black] (11.1,10.55) circle (0.0) node [left] {\Large{$\overleftarrow{A_4}\,$}};

\end{tikzpicture}
}
\caption[b]{}
\label{fig:pi2}
\end{subfigure}
\qquad\qquad
\begin{subfigure}[b]{0.35\textwidth}

\scalebox{0.5}
{

\begin{tikzpicture}
[line width = 0pt,
vtx/.style={},
]
    \node[vtx] (v1) at (1,4) {};
    \node[vtx] (v2) at (2,5) {};
    \node[vtx] (v3) at (3,6) {};
    \node[vtx] (v4) at (4,9) {};
    \node[vtx] (v5) at (5,8) {};
    \node[vtx] (v6) at (6,7) {};
    \node[vtx] (v7) at (7,1) {};
    \node[vtx] (v8) at (8,2) {};
    \node[vtx] (v9) at (9,3) {};




%
   \draw[line width=1.25mm, color=black]  (v1) -- (v3);
   \draw[line width=1.25mm, color=black]  (v4) -- (v6);
   \draw[line width=1.25mm, color=black]  (v7) -- (v9);

   \fill[fill=black] (2,5.2) circle (0.0) node [left] {\Large{$-\overleftarrow{A_1}\ $}};
   \fill[fill=black] (4.9,8.3) circle (0.0) node [right] {\ \Large{$\overleftarrow{A_2}$}};
   \fill[fill=black] (8.1,2.6) circle (0.0) node [left] {\Large{$-\overleftarrow{A_3}\!$}};

   \draw (2.5,0.5) -- (8.5,0.5) -- (8.5, 10.5) -- (2.5, 10.5) -- (2.5, 0.5);
   \node[text width=10cm] at (7.65, 9.8)
    {\Large{\sffamily{Window of width $4k-2$}}};

   \draw[line width=1.5mm, color=blue] (2.5,5.5) -- (3,6);
   \draw[line width=1.5mm, color=blue] (4,9) -- (4.5,8.5);
   \draw[line width=1.5mm, color=red] (4.5,8.5) -- (5,8);
   \draw[line width=1.5mm, color=blue] (5,8) -- (6,7);
   \draw[line width=1.5mm, color=red] (7,1) -- (8,2);
   \draw[line width=1.5mm, color=blue] (8,2) -- (8.5,2.5);

   \node at (2.9,5.4) {\textcolor{blue}{\Large{$D$}}};
   \node at (4,8.4) {\textcolor{blue}{\Large{$D$}}};
   \node at (4.5,8) {\textcolor{red}{\Large{$L$}}};
   \node at (5.2,7.25) {\textcolor{blue}{\Large{$D$}}};
   \node at (7.65,1.2) {\textcolor{red}{\Large{$L$}}};
   \node at (8.25,1.75) {\textcolor{blue}{\Large{$D$}}};

\end{tikzpicture}

}

\caption[b]{}
\label{fig:piDL}
\end{subfigure}

\caption{\subref{fig:pi2} The shape of $\pi=\pi^{(2)}$ from the proof of Proposition~\ref{infty}; \subref{fig:piDL} Defining $D$ and $L$ in the proof of Proposition~\ref{infty}.}
\label{fig:pi}

\end{figure}

We will now show that $\pi:=\pi^{(k)}$ contains no tight 2-twins of length $2k-1$ or  $2k$. By symmetry, it suffices to consider only blocks (`windows') of length $4k-2$ and $4k$, which begin at one of the first $2k-1$ elements of $\pi$. Let such a window begin at the $s$-th element of $(-\overleftarrow{A_1})$, that is at $\pi(s)$, $1\le s\le 2k-1$. It then may stretch over the entire block $\overleftarrow{A_2}$ and, for odd twins, possibly, over some initial segment of $(-\overleftarrow{A_3})$. For even twins, when $s=2k-1$, the window reaches even the first element of $\overleftarrow{A_4}$.

Suppose there are  tight 2-twins of length $2k-1$ in $\pi$ beginning at $\pi(s)$. They yield a partition of the set  $\{\pi(s),\pi(s+1),\dots,\pi(4k-2+s-1)\}=D\cup L$ of length $4k-2$, where $|D|=|L|=2k-1$ (see Figure~\ref{fig:piDL}). Clearly, none of the twins can coincide with $\overleftarrow{A_2}$. For $s$ odd, since then also $2k-1-(s-1)=2k-s$ is odd, w.l.o.g., the twin on $D$ (call it \emph{Daphne}, cf. \cite{Gram}) begins with an increasing segment longer than the initial increasing segment of the $L$-twin (call it \emph{Laurel}), a contradiction with Daphne and Laurel being twins. 
For $s$ even, we look at the other end, that is, at the first $s-1$ elements of $(-\overleftarrow{A_3})$, where again, due to the oddity of $s-1$, Daphne, say, captures more elements than Laurel. This means, however, that Daphne ends with a longer increasing segment than Laurel, a contradiction, again.

It remains to exclude tight 2-twins of length $2k$ in $\pi$. Suppose that there are such twins, $D$ and $L$, or Daphne and Laurel. We are now looking at a window on positions $\{s,s+1,\dots,4k+s-1\}=D\cup L$, $1\le s\le 2k-1$, of length $4k$,  where $|D|=|L|=2k$. The argument is similar to that for twins of length $2k-1$.  Clearly, neither Daphne or Laurel can contain an entire block $\overleftarrow{A_2}$ or $(-\overleftarrow{A_3})$.  As before, for $s$ odd, Daphne begins with an increasing segment longer than the initial increasing segment of Laurel, a contradiction.
For $s$ even, the ending segment consisting of the first $s+1$ elements of $(-\overleftarrow{A_3})$ has an odd length at least 3 and, again,  Daphne, say, ends with a longer increasing segment than Laurel which is a contradiction.
\qed

Now we show how the above construction can be generalized to yield a similar result for any $r\ge 3$.
\begin{prop}\label{rinfty} For all $r,k\ge3$, we have
	$f(r,k)=\infty$.
\end{prop}
\proof We use previous constructions with blocks of length $rk-1$, $k\ge2$, obtaining a permutation $\pi_r^{(k)}$. Let us denote its consecutive blocks of length $rk-1$ by $B_1,B_2,\dots$, that is, $B_1=-\overleftarrow{A_1}$, $B_2=\overleftarrow{A_2}$, etc. We will show that  $\pi_r^{(k)}$ does not contain tight $r$-twins of length $\ell\in\{2k-1,2k\}$. Suppose the opposite and let $T_1,\dots,T_r$ be tight $r$-twins of length $\ell$ in $\pi_r^{(k)}$.

W.l.o.g., consider a window $W=T_1\cup\cdots\cup T_r$ of length $r\ell$ beginning at the $s$-th element of $B_1$. Then, since $\ell\ge 2k-1$, $s\ge1$ and $r\ge2$,
$$|W\cap B_2|\ge\min\{2\ell-(rk-s),rk-1\}\ge\min\{2rk-r-(rk-s),rk-1\}\ge rk-r+1.$$
 Also, $W\cap B_4=\emptyset$, except for $\ell=2k$ and $s=rk-1$, when $|W\cap B_4|=1$. By taking the average, there is $T_{i_0}$ with $|T_{i_0}\cap B_2|=\min_{i}|T_{i}\cap B_2|\le k-1$ and $T_{i_1}$ with $|T_{i_1}\cap B_2|=\max_{i}|T_{i}\cap B_2|\ge k$. Moreover, if  for no $i$, $|T_{i}\cap B_2|\ge k+1$, then for all $i$, $|T_{i}\cap B_2|\ge k-1\ge 1$.

We claim that for all $i$, $|T_{i}\cap B_2|\ge 1$. Suppose  that $T_{i_0}\cap B_2=\emptyset$. Then, in view of the above, $|T_{i_1}\cap B_2|\ge k+1\ge3$. This means that $T_{i_1}$ has a decreasing segment of length at least 3, while $T_{i_0}$ does not, as $T_{i_0}\subset B_1\cup B_3\cup\{f\}$, where $f$ is the first element of $B_4$.

 Also, for all $i$, $|T_{i}\cap B_1|\ge 1$. Indeed, otherwise there would be a twin which begins with a decreasing segment and a twin which begins with an increasing segment, a contradiction. Finally, compare $T_{i_0}$ with $T_{i_1}$. There is in $T_{i_1}$ a decreasing segment of length $k$ with all elements larger than the first element of $T_{i_1}$. On the other hand, there is not such a segment in $T_{i_0}$ (there might be a decreasing segment of length $k$ which, however, ends in $B_3$, thus, below the first element of $T_{i_0}$ which belongs to $B_1$). This is a contradiction, again, and the proof is completed. \qed

\bigskip

Proposition 3.7 in \cite{DGR} together with permutation $\pi_2=(1,4,3,2,5)$ show that $f(2,2)=6$.
Thus, in view of Proposition \ref{infty}, $f(2,k)$ is determined for all $k\ge2$.
Let us now focus on $f(r,2)$, $r\ge3$. For simplicity, we will call tight $r$-twins of length 2 just \emph{$r$-twins}.
Besides $f(2,2)=6$, we also know that $f(3,2)=12$. Indeed, permutation
$$\pi_3=(11,2,3,8,7,6,5,4,9,10,1)$$
of length 11 is free of  3-twins, while  a computer verification of all permutations of length 12 reveals that each one of them contains  3-twins. For  $r\ge4$, however, we only have a lower bound on $f(r,2)$, which is quadratic in $r$.

\begin{prop}\label{ttr2}
	For every $r\ge3$ we have $f(r,2)\ge r(r+5)-12$.
\end{prop}
\proof We begin by constructing a suitable permutation of length $r(r+5)-13$.

For $r=3$, we have already presented above  permutation $\pi_3$  of the required length $11=3(3+5)-13$.
Fix $r\ge4$ and consider the following sequence of sequences of consecutive  integers of (mostly) diminishing length: $A_0=(1,\dots,r-1)$, $A_1=(r,\dots,2r-2)$, $A_2=(2r-1,\dots,3r-4)$, $A_3=(3r-3,\dots,4r-7)$, $\dots$, $A_{r-1}=(\binom r2+r-1)$, $A_{r}=(\binom r2+r)$, $\dots$, $A_{3r-7}=(\binom r2+3r-7)$. Note that the first two sequences have the same length $r-1$, then each next one is shorter by one, and, finally $|A_{r-1}|=|A_r|=\cdots=|A_{3r-7}|=1$. In total, their concatenation makes up the sequence $(1,\dots, \binom r2+3r-7)$. Define permutation
$$\pi'_r=((-1)^{r-1}\overleftarrow{A_{3r-7}})\cdots \overleftarrow{A_2}\;(-\overleftarrow{A_1})\;\overleftarrow{A_0}\;0\;(-A_0)\;A_1\;(-A_2)\cdots (-1)^{r}A_{3r-7}$$
of length
$$2\times\left(\binom r2+3r-7\right)+1=r(r+5)-13.$$
(Of course, for singleton classes the overhead arrow is redundant.)
Further, let $\pi_r$ be the reduced form of $\pi'_r$, obtained, simply, by adding $\binom r2+3r-6$ to all elements of $\pi'_r$.
For instance,
$$\pi'_4=(-11,10,-9,8,7,-6,-5,-4,3,2,1,0,-1,-2,-3,4,5,6,-7,-8,9,-10,11)$$
becomes
$$\pi_4=(1,22,3,20,19,6,7,8,15,14,13,12,11,10,9,16,17,18,5,4,21,2,23)$$
(see Figure~\ref{fig:pi4}).
One more example:
\begin{align*}
\pi_5=(37,2,35,4,&33,6,7,30,29,28,11,12,13,14,23,22,21,20,
\\&19,18,17,16,15,24,25,26,27,10,9,8,31,32,5,34,3,36,1).
\end{align*}

\begin{figure}

\scalebox{0.45}
{

\begin{tikzpicture}
[line width = .5pt,
vtx/.style={circle,draw,black,very thick,fill=black, line width = 3pt, inner sep=2pt},
]

    \node[vtx] (v1) at (1,1) {};
    \node[vtx] (v2) at (2,22) {};
    \node[vtx] (v3) at (3.5,3) {};
    \node[vtx] (v4) at (4,20) {};
    \node[vtx] (v5) at (5,19) {};
    \node[vtx] (v6) at (6,6) {};
    \node[vtx] (v7) at (7,7) {};
    \node[vtx] (v8) at (8,8) {};
    \node[vtx] (v9) at (9,15) {};
    \node[vtx] (v10) at (10,14) {};
    \node[vtx] (v11) at (11,13) {};
    \node[vtx] (v12) at (12,12) {};
    \node[vtx] (v13) at (13,11) {};
    \node[vtx] (v14) at (14,10) {};
    \node[vtx] (v15) at (15,9) {};
    \node[vtx] (v16) at (16,16) {};
    \node[vtx] (v17) at (17,17) {};
    \node[vtx] (v18) at (18,18) {};
    \node[vtx] (v19) at (19,5) {};
    \node[vtx] (v20) at (20,4) {};
    \node[vtx] (v21) at (21,21) {};
    \node[vtx] (v22) at (22,2) {};
    \node[vtx] (v23) at (23,23) {};

    \draw[line width=0.75mm, color=black]  (v1) -- (v2) -- (v3) -- (v4) -- (v5) -- (v6) -- (v7) -- (v8) -- (v9) -- (v10) -- (v11) -- (v12) -- (v13) -- (v14) -- (v15) -- (v16) -- (v17) -- (v18) -- (v19) -- (v20) -- (v21) -- (v22) -- (v23);

   \draw[color=gray] (0,1.5) -- (24,1.5);
 \draw[color=gray] (0,2.5) -- (24,2.5);
   \draw[color=gray] (0,3.5) -- (24,3.5);
   \draw[color=gray] (0,5.5) -- (24,5.5);
   \draw[color=gray] (0,8.5) -- (24,8.5);
   \draw[color=gray] (0,12) -- (24,12);
   \draw[color=gray] (0,15.5) -- (24,15.5);
   \draw[color=gray] (0,18.5) -- (24,18.5);
   \draw[color=gray] (0,20.5) -- (24,20.5);
   \draw[color=gray] (0,21.5) -- (24,21.5);
   \draw[color=gray] (0,22.5) -- (24,22.5);

   \fill[fill=black] (v12) circle (0.1) node [right] {\Large{$\ 0$}};
   \fill[fill=black] (v10) circle (0.1) node [right] {\Large{$\ \overleftarrow{A_0}$}};
   \fill[fill=black] (v14) circle (0.1) node [left] {\Large{$-A_0\ $}};
   \draw[line width=1.25mm, color=black]  (v9) -- (v11);
   \draw[line width=1.25mm, color=black]  (v13) -- (v15);

   \fill[fill=black] (v7) circle (0.1) node [right] {\Large{$\ -\overleftarrow{A_1}$}};
   \draw[line width=1.25mm, color=black]  (v6) -- (v8);

  \fill[fill=black] (v17) circle (0.1) node [left] {\Large{$A_1\ $}};
  \draw[line width=1.25mm, color=black]  (v16) -- (v18);

  \fill[fill=black] (4.3,19.9) circle (0.0) node [right] {\Large{$\; \overleftarrow{A_2}$}};
  \draw[line width=1.25mm, color=black]  (v4) -- (v5);

  \fill[fill=black] (19.6,4.3) circle (0.0) node [left] {\Large{$\; -A_2$}};
  \draw[line width=1.25mm, color=black]  (v19) -- (v20);

  \fill[fill=black] (v3) circle (0.1) node [right] {\Large{$\; -\overleftarrow{A_3}$}};
  \fill[fill=black] (v2) circle (0.1) node [right] {\Large{$\; \overleftarrow{A_4}$}};
  \fill[fill=black] (v1) circle (0.1) node [right] {\Large{$\; -\overleftarrow{A_5} = (\pi(1))$}};

  \fill[fill=black] (v21) circle (0.1) node [left] {\Large{$A_3\ $}};
  \fill[fill=black] (v22) circle (0.1) node [left] {\Large{$-A_4\ $}};
  \fill[fill=black] (v23) circle (0.1) node [left] {\Large{$A_5 = (\pi(23))\ $}};

\end{tikzpicture}

}

\caption{The shape of $\pi=\pi_4$ (and as well as $\pi'_4$) from the proof of Proposition~\ref{ttr2}.}
\label{fig:pi4}

\end{figure}
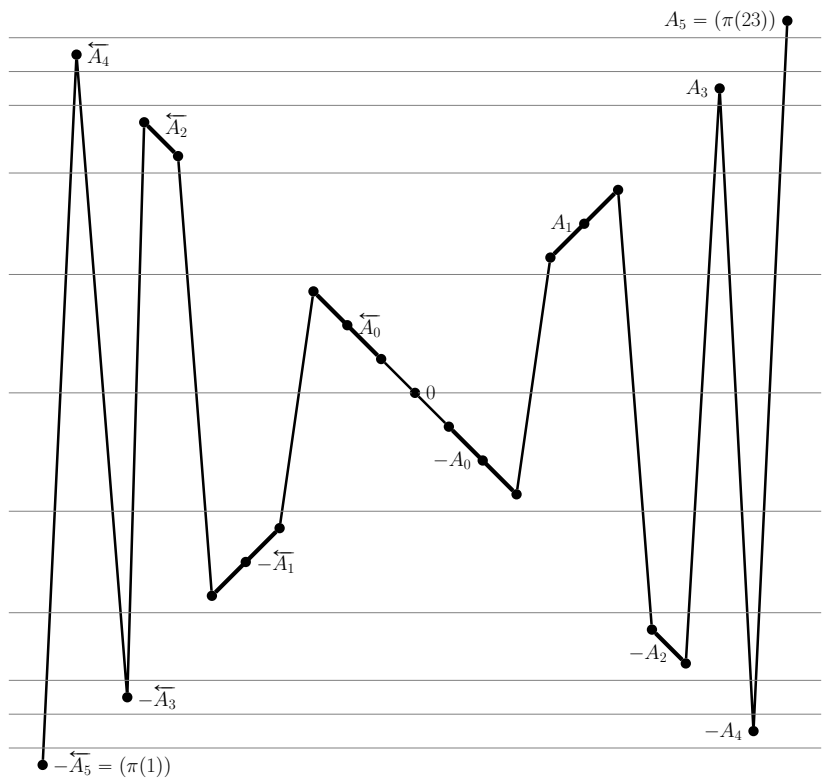

In general, $\pi_r$ consists of a decreasing sequence of length $2r-1$ located in the middle which is extended in both directions by alternatingly increasing and decreasing sequences of  lengths getting shorter by one in each step, except for the very end where, on each side, we have a zigzag pattern of singletons of length $2r-5$. Moreover, what is crucial here, each next monotone segment (to the left or to the right) is entirely below or entirely above all previous elements appearing on the same side.

We are going to prove that there are no tight $r$-twins of length 2 (shortly, $r$-twins) in~$\pi_r$.   For the proof we need to distinguish two types of $r$-twins. We call $r$-twins  \emph{increasing} if they are similar to $(1,2)$ and \emph{decreasing} if they are similar to $(2,1)$.
Suppose that there are  $r$-twins in $\pi_r$.
As they occupy a block (window) of length $2r$, let us consider all possible locations of it. Due to symmetry, it suffices to consider windows with the right end belonging to the right half of $\pi_r$, that is, to be to the right of 0 in $\pi'_r$ (for the ease of description we will look at $\pi'_r$, not $\pi_r$, which is, of course, equivalent).

Let $w$ be the rightmost element of the window and assume first that it belongs to $A_m$, $m$ odd. Then, since the last element of the window is the largest one, the $r$-twins have to be increasing. Thus, the elements of the previous block, $(-A_{m-1})$ have to be all paired with the elements of $A_m$. If $m\le r-1$, then, $|A_{m-1}|<|A_m|$, and we get a contradiction. If $w\in A_r\cup\cdots\cup A_{3r-7}$, i.e., it belongs to the zigzag of singletons at the end,  the window still contains $A_{r-3}\cup A_{r-2}$. Then each singleton has to be paired with the next one, and the three elements of $A_{r-3}$ have to be paired with the two elements of $A_{r-2}$, a contradiction again. For $m$ even, the situation is symmetrical: the last element is the smallest, so the $r$-twins have to be decreasing, which again leads to a contradiction.

It remains to consider the case when $w$ belongs to $(-A_0)$. Then the window contains a decreasing sub-sequence of length at least $r+1$ (all its elements but those sitting in $(-A_1)$). But, clearly, no two elements of a decreasing sub-sequence can be paired with each other in increasing twins.
We, again, arrive at a contradiction which completes the proof. \qed

\section{Block, tight, and block-tight twins in random permutations}\label{Sr} In the previous sections we used a random permutation $\Pi_n$ as a tool of the probabilistic method to estimate $bt^{(r)}(n)$ and $tt^{(r)}(n)$.
Now we are interested in the longest length of block, tight, and block-tight $r$-twins \emph{in} a random permutation. We say that an event $\cE_n$ in the uniform probability space of all $n!$ permutations of $[n]$ holds \emph{asymptotically almost surely}, or \emph{a.a.s.}, for short, if $\PP(\cE_n)\to1$, as $n\to\infty$.

\subsection{Block twins} The next result  shows that  the maximum length of block $r$-twins in $\Pi_n$ is a.a.s.~just a little bit greater than in the worst case and the difference diminishes with $r$ increasing (cf. Theorem \ref{Theorem Block Twins}).

\begin{thm}\label{Theorem Block Twins Random}
	For a random $n$-permutation $\Pi_n$, a.a.s.~we have
	\[
	bt^{(r)}(\Pi_n) = \left(\frac r{r-1}+o(1)\right) \frac{\log n}{\log \log n}.
	\]
\end{thm}

\begin{proof} For an integer $k\ge2$,
	recall from the proof of Theorem \ref{Theorem Block Twins} that for a given $k$-spread $r$-tuple of indices $i_1,\dots,i_r$, $\cE_{i_1,\dots,i_r}$ denotes the event that segments $(\Pi(i_j),\Pi(i_j+1)),\dots,\Pi(i_j+k-1))$, $j=1,\dots,r$,  form block $r$-twins.
	Let $X_{i_1,\dots,i_r}$ be the indicator random variable of the event $\cE_{i_1,\dots,i_r}$, that is, $X_{i_1,\dots,i_r}=1$ if $\cE_{i_1,\dots,i_r}$ holds; otherwise $X_{i_1,\dots,i_r}=0$. Set $X(k)=\sum X_{i_1,\dots,i_r}$, where the sum extends over all $k$-spread $r$-tuples $i_1,\dots,i_r$.
	
	By \eqref{1|k!}, we have $\Prob(X_{i_1,\dots,i_r}=1)=\PP(\cE_{i_1,\dots,i_r})=k!^{-(r-1)}$ and thus $\E(X(k)) = \Theta(n^rk!^{-(r-1)})$, where the hidden constant is less than one.
	Let
	$$k^+ = \left\lceil  \frac{(1+\epsilon_n)r\log n}{(r-1)\log\log n}\right\rceil,\quad\mbox{where}\quad\frac{\log\log n}{\log\log\log n}\ll\epsilon_n=o(1).$$ Then, with $c_r=\tfrac{e(r-1)}r$,
	$$\E X(k^+)\le n^r(k^+)!^{-(r-1)}\le\left(\frac e{k^+}\right)^{k^+(r-1)}n^r\le\left(\frac{c_r\log\log n}{\log n}\right)^{\frac{(1+\epsilon_n)r\log n}{\log\log n}}n^r\to0,$$
	because, after taking the logarithm,
	$$\frac{(1+\epsilon_n)r\log n}{\log\log n}(\log c_r+\log\log\log n-\log\log n)+r\log n\to-\infty.$$
	Hence, by Markov's inequality, a.a.s., $X(k^+)=0$, that is, $bt^{(r)}(\Pi_n)<k^+$.

	We will establish a matching lower bound on $bt^{(r)}(\Pi_n)$ by the second moment method. Set $k^- = \left\lfloor  \frac{r\log n}{(r-1)\log\log n}\right\rfloor$. Then, $\E X(k^-)\to\infty$, since
	$$\frac{n^r}{(k^-)!^{r-1}}\ge\left(\frac1{k^-}\right)^{(r-1)k^-}n^r\ge\left(\frac{(r-1)\log\log n}{r\log n}\right)^{\frac{r\log n}{\log\log n}}n^r\to\infty,$$
	because $$\frac{r\log n}{\log\log n}\big{(}\log((r-1)/r)+\log\log\log n-\log\log n\big{)}+r\log n\to\infty.$$
	By the same kind of calculations, we also have,
	\begin{equation}\label{samekind}
	\frac{n^r}{(\log n)^r(k^-)!^{r-1}}\to\infty\quad\mbox{and}\quad n/(k^-)!=o(1).
	\end{equation}
	
	Now, we turn to estimating $Var(X(k^-))$. For two $k^-$-spread $r$-tuples of indices, $i_1,\dots,i_r$ and $i'_1,\dots,i'_r$, let us analyze the covariance $Cov(X_{i_1,\dots,i_r},X_{i'_1,\dots,i'_r})$. Set $A_j=\{i_j,\dots,i_j+k-1\}$ and $A'_j=\{i'_j,\dots,i'_j+k-1\}$, $j=1,\dots,r$. Let $0\le s \le r$ be the largest integer such that there are indices $1\le j_1<\cdots<j_s\le r$ with
	\begin{equation}\label{s}
	(A_{j_1}'\cup\cdots\cup A_{j_s}')\cap(A_1\cup\dots\cup A_r)=\emptyset.
	\end{equation}
	Then, by Lemma \ref{ABCD} with $t=2$, for $s\in\{r-1,r\}$, due to independence, $Cov(X_{i_1,\dots,i_r},X_{i'_1,\dots,i'_r})=0$, while for $s\le r-2$, using also \eqref{1|k!},
	$$Cov(X_{i_1,\dots,i_r},X_{i'_1,\dots,i'_r})\le\PP(X_{i_1,\dots,i_r}=X_{i'_1,\dots,i'_r}=1)\le\frac1{(k^-)!^{s+r-1}}.$$
	Moreover, for $s\le r-2$ and a given $k^-$-spread $r$-tuple $i_1,\dots,i_r$, the number of  $k^-$-spread $r$-tuples $i'_1,\dots,i'_r$ satisfying \eqref{s} is $o\left(n^s(\log n)^{r-s}\right)$. Indeed, for each $j\not\in \{j_1,\dots,j_s\}$ there are no more than $2rk^-=o(\log n)$ choices for placing $i_j$, while for $j\in \{j_1,\dots,j_s\}$ ``the sky's the limit''.

	Hence,
	$$Var(X(k^-))=\sum_{i_1,\dots,i_r}\sum_{i'_1,\dots,i'_r}Cov(X_{i_1,\dots,i_r},X_{i'_1,\dots,i'_r})
	=o\left(n^r\sum_{s=0}^{r-2}\frac{n^s(\log n)^{r-s}}{(k^-)!^{s+r-1}}\right)$$ and, by Chebyshev's inequality and \eqref{samekind}
	$$\PP(X(k^-)=0)\le\frac{Var(X(k^-))}{(\E X(k^-))^2}=o\left(\frac{(\log n)^r(k^-)!^{r-1}}{n^r}\sum_{s=0}^{r-2}\left(\frac n{(k^-)!}\right)^s\right)=o(1).$$
	
\end{proof}

\subsection{Tight and block-tight twins} It turns out that the longest tight and block-tight $r$-twins in a random permutation have asymptotically the same length. To see the reason, let $Y(k)$ and $Z(k)$ denote, resp., the number of tight and block-tight $r$-twins of length $k$ in $\Pi_n$. Then
$$\E Y(k)=(n-rk+1)\times\frac1{r!}\binom{rk}{k,\dots,k}\times\frac1{k!^{r-1}}\quad\mbox{and}\quad\E Z(k)=(n-rk+1)\times\frac1{k!^{r-1}},$$ and the extra factor in $\E Y(k)$, counting the partitions of a block of length $rk$ into $r$ blocks of length $k$, turns out to be of an negligible order of magnitude.

We put these two results under one theorem, because they have a common proof. Indeed,  as every block-tight $r$-twins are also tight, $Z(k)\le Y(k)$, so it will be sufficient to bound $\PP( Y(k)>0)$ and  $\PP( Z(k)=0)$ only. This is quite fortunate, as estimating $\PP( Y(k)=0)$ seems to be much harder.
In fact, the estimates needed in the proof of Theorem \ref{bttRandom} below become very similar to, and even easier than, those in the proof of Theorem \ref{Theorem Block Twins Random}. There is one twist, however.
Since the property of possessing tight, as well as block-tight, $r$-twins of length $k$ is not monotone in $k$, to prove the upper bound we need to estimate not just $\E X(k^+)$, but $\sum_{k\ge k^+}\E X(k^+)$.

Note that, roughly, the longest tight and block-tight $r$-twins in a random permutation are $r$ times shorter than largest block $r$-twins.

\begin{thm}\label{bttRandom}
	For a random $n$-permutation $\Pi_n$, a.a.s.~we have
	\[
	tt^{(r)}(\Pi_n) = \left(\frac 1{r-1}+o(1)\right) \frac{\log n}{\log \log n}=btt^{(r)}(\Pi_n).
	\]
\end{thm}

\proof For $k\ge \log n$, with $c_r=er^{r/(r-1)}$,
$$\E Y(k)\le \frac{n r^{rk}}{k!^{r-1}}\le n\left(\frac{c_r}k\right)^{(r-1)k}\le n\left(\frac{c_r}{\log n}\right)^{(r-1)\log n}=o(n^{-1}),$$
and so
$$\sum_{\log n \le k\le n/r}\E Y(k)=o(1).$$

To deal with the lower range of $k$, let
$$k^+ = \left\lceil  \frac{(1+\epsilon_n)\log n}{(r-1)\log\log n}\right\rceil\quad\mbox{where}\quad\frac{\log\log\log n}{\log\log n}\ll\epsilon_n=o(1).$$ Then,
for every $ k^+\le k\le \log n$,
with $c'_r=(r-1)c_r$,
$$\E Y(k)\le n\left(\frac{c_r}k\right)^{(r-1)k}
= n\left(\frac{c'_r}{(r-1)k}\right)^{(r-1)k}
\le n\left(\frac{c'_r\log\log n}{\log n}\right)^{\frac{(1+\epsilon_n)\log n}{\log\log n}}\le n^{-\epsilon_n/2},$$
because, after taking the logarithm,
$$\log n-\frac{(1+\epsilon_n)\log n}{\log\log n}\left(\log\log n-\log\log\log n-\log c'_r\right) + \frac{\epsilon_n}{2}\log n \to-\infty.$$
Hence,
$$\sum_{k^+\le k\le\log n}\E Y(k)=O\left((\log n)n^{-\epsilon_n/2}\right)=o(1)$$
and, consequently,
by Markov's inequality,
$$\PP(\exists k\ge k^+:\;Y(k)>0)\le\sum_{k^+\le k\le n/r}\E Y(k)=o(1),$$
that is, a.a.s., $btt^{(r)}(\Pi_n)\le tt^{(r)}(\Pi_n)<k^+$.

We now establish a matching lower bound on $btt^{(r)}(\Pi_n)$ by the second moment method.  Set $k^- = \left\lfloor  \frac{\log n}{(r-1)\log\log n}\right\rfloor$. Then, $\E Z(k^-)\to\infty$, since
$$\E Z(k^-)\ge\frac{n}{2(k^-)!^{r-1}}\ge  n \left(\frac1{k^-}\right)^{(r-1)k^-}\ge n\left(\frac{\log\log n}{\log n}\right)^{\frac{\log n}{\log\log n}}\to\infty,$$
because $$\log n-\frac{\log n}{\log\log n}(\log\log n-\log\log\log n)\to\infty.$$

To bound the variance of $Z(k^-)$, for every block $B$ of length $rk$ in $[n]$ denote by $I_B$ the indicator random variable that $\Pi_n$ spans on $B$ block-tight $r$-twins and observe that as a simple consequence of Lemma \ref{ABCD} (case $s=r$), $I_B$ and $I_{B'}$ are independent whenever $B\cap B'=\emptyset$.
For $B\cap B'\neq\emptyset$ we will trivially bound $Cov(I_B,I_{B'})\le \PP(I_B=1)=(k^-)!^{-(r-1)}$. Also observe that for a fixed $B$ the number of choices of $B'$ satisfying $B\cap B'\neq\emptyset$ is $O(1)$. Thus,
$$
Var(Z(k^-))=O\left(\frac n{(k^-)!^{r-1}}\right)=\Theta\left(\E Z(k^-)\right)
$$
and, consequently,
$$
\PP(Z(k^-)=0)\le\frac{Var(Z(k^-))}{\E Z(k^-)^2}=O\left(\frac1{\E Z(k^-)}\right)=o(1),
$$
that is, a.a.s.~$tt^{(r)}(\Pi_n)\ge btt^{(r)}(\Pi_n)\ge k^-$. \qed

\section{A third point of view}\label{third}

So far we considered two scenarios with respect to the three parameters $n$ - the length of permutation, $r$ - the multiplicity of twins, and $k$ - the length of twins. In the main object of interest in this paper, parameters $bt^{(r)}(n)$, $tt^{(r)}(n)$, etc., we fixed $r$, let $n\to\infty$, and asked for the largest $k$. When studying function $f(r,k)$ in Subsection \ref{around}, we fixed $r$ and $k$ and asked for the smallest $n$. In this section we consider a third ``point of view'', where we fix $k$ and $n$ (or let $n\to\infty$) and ask for the largest $r$.

 Given $k$ and a permutation $\pi$, let $r_{bt}^{(k)}(\pi)$, resp. $r_{tt}^{(k)}(\pi)$, be the largest $r$ such that $\pi$ contains block, resp. tight, $r$-twins of length $k$. (To make this parameter well defined we allow $r=1$.)
 Define  $r_{bt}^{(k)}(n)$, resp. $r_{tt}^{(k)}(n)$, as the minimum of $r^{(k)}(\pi)$, resp. $r_{tt}^{(k)}(\pi)$, over all $n$-permutations $\pi$.

 It follows from the pigeonhole principle (cf.~the proof of Theorem~\ref{Theorem Block Twins}) that $r_{bt}^{(k)}(n)\ge \lfloor(n/k-1)/k!\rfloor+1$, which for $k=2$ can be pinpointed to $r_{bt}^{(2)}(n)=\lfloor(n+2)/4\rfloor$  by considering permutation $\pi$ with $\pi(1)<\pi(2)<\pi(3)>\pi(4)>\pi(5)<\pi(6)<\pi(7)>\pi(8)\cdots$. Also, by~\cite{DGR_weak}, Thm.~1.2,  a.a.s.~$r_{bt}^{(2)}(\Pi_n)\ge(7/20+o(1))n$, as one can take every other pair from each alternating sequence in $\Pi_n$.

From Propositions \ref{infty} and \ref{rinfty}, it follows that $r_{tt}^{(k)}(n)=1$ for $k\ge3$, while
from the lower bound  on $f(r,2)$ in Proposition \ref{ttr2} we have $r_{tt}^{(2)}(n)=O(\sqrt n)$.

As far as $r_{tt}^{(k)}(\Pi_n)$ is concerned we have a complete solution for $k=2$ only. Namely, we show that a random permutation $\Pi_n$, $n$ even, contains a.a.s.~$n/2$-twins of length 2, an optimal result comparable with the presence of a perfect matching in a graph. (In fact, we do use Hall's Theorem in the proof.) Although, we do not specify it, the proof yields the existence of $n/2$-twins of length 2 similar to $(1,2)$ as well as similar to $(2,1)$.

\begin{thm}\label{r2Pi}
A.a.s.~$r_{tt}^{(2)}(\Pi_n)=n/2$.
\end{thm}
\proof Set $n=2r$ and consider an auxiliary bipartite graph $B$ between $U:=\{1,\dots,r\}$ and $W:=\{r+1,\dots,2r\}$ where $ij\in B$ if $\Pi_n(i)<\Pi_n(j)$. As $\PP(ij\in B)=1/2$, for a fixed $i=1,\dots,2r$, the random variable $X(i):=\deg_B(i)$ has expectation $\E X(i)=r/2$. Similarly, for $\{i,j\}\in\binom{U}2$ and $\ell\in W$, or $\{i,j\}\in\binom{W}2$ and $\ell\in U$, we have $\PP(i\ell\in B,\;j\ell\in B)=1/3$. Indeed, out of 6 possible relative permutations of $\Pi(i),\Pi(j)$, and $\Pi(\ell)$, exactly 2 are such that $\Pi(i)<\Pi(\ell)$ and $\Pi(j)<\Pi(\ell)$. Thus, the random variable $Y(i,j):=|\{\ell: i\ell\in B,\:j\ell\in B\}|$, or the co-degree of $i,j$ in $B$,  has expectation $\E Y(i,j)=r/3$.

 Observe that both, $X(i)$ and $Y(i,j)$, satisfy the Lipschitz condition for permutations with $c=1$, that is swapping around two values of a permutation changes the value of the function by at most 1. Thus, one may apply the Azuma-Hoeffding inequality for random permutations (see, e.g., Lemma 11 in~\cite{FP} or  Section 3.2 in~\cite{McDiarmid98}, or Thm. 2.6 in \cite{DGR}) and, using also the union bound, conclude that a.a.s.~for all $i,j$ we have $|X(i)-r/2|\le r^{2/3}$ and $|Y(i,j)-r/3|\le r^{2/3}$, for $r$ large.

We intend to apply Hall's Marriage Theorem to $B$. Recall that if the Hall's condition, $|N(S)|\ge |S|$, is violated for some $S$, then it is also violated by some $S'$ such that $|S'|\le \lceil r/2\rceil$. Thus, it is enough to check Hall's condition for, say, $|S|\le \tfrac7{12}r$, $r$ large.
If $S=\{i\}$, then, trivially,  $|N(S)|=X(i)\ge1$. If $2\le |S|\le \tfrac7{12}$, then for any $\{i,j\}\subset S$, we have
$$|N(S)|\ge |N(\{i,j\}|=X(i)+X(j)-Y(i,j)\ge\frac23r-r^{2/3}>|S|.$$
Thus, a.a.s., there is in $B$ a perfect matching which corresponds to tight $r$-twins of length 2 in $\Pi_{2r}$, similar to $(1,2)$. To obtain the other type, $(2,1)$, apply the same proof with the definition of $B$ changed to $ij\in B$ if $\Pi_n(i)>\Pi_n(j)$.
\qed

\section{Concluding Remarks}\label{cr}

We conclude with some open problems for future considerations.
In Theorem \ref{Theorem Tight Twins} we proved that $tt^{(r)}(n)\le 15r$.

\begin{prob}\label{P61}
	Is it true that $tt^{(r)}(n)\leqslant c$ for some absolute constant $c$?
\end{prob}
As mentioned after the proof of Theorem \ref{Theorem Tight Twins}, due to the weakness of the bound on $\PP(\cA_K)$, one will probably need other tools than the Local Lemma. The above probability is, however,  of its own interest. To extract the essence of the problem, let
 $n=kr$ and let $Q^{(r)}(k)$ denote the number of permutations of $[n]$ that are tight $r$-twins of length~$k$. From the proof of Theorem \ref{Theorem Tight Twins} we know already that for large $k$ with respect to $r$, $Q^{(r)}(k)/(rk)!\le p_k\to0$ exponentially fast. How about the other way around, that is, when $k$ is fixed and $r\to\infty$?

\begin{prob}\label{P62}
	Determine the asymptotic order of $Q^{(r)}(k)$ for every fixed  $k\geqslant 2$ and $r\to\infty$.
\end{prob}
Recall parameters $r_{bt}^{(k)}(\pi)$ and $r_{tt}^{(k)}(\pi)$ introduced in Section \ref{third} and note that\linebreak $\PP(r_{tt}^{(k)}(\Pi_n)=n/k)=Q^{(n/k)}(k)/n!$.
Thus, if $Q^{(r)}(k)\sim (rk)!$, this would mean that a random permutation $\Pi_n$, a.a.s.~contains $r$-twins (of length $k=n/r$) which cover it entirely, or $r^{(k)}_{tt}(\Pi_n)=n/k$.
In Section \ref{third} we proved it only for $k=2$ (cf. Theorem \ref{r2Pi}).


\begin{prob}
Find  asymptotic distributions of $r_{bt}^{(k)}(\Pi_n)$ and $r_{bt}^{(k)}(\Pi_n)$ for every fixed $k$ and $n\to\infty$.
\end{prob}

In Proposition \ref{ttr2} we showed a lower bound on $f(r,2)$.
\begin{prob}
	For  $r\ge4$, find an upper bound on $f(r,2)$, or prove that $f(r,2)=\infty$.
\end{prob}

\end{document}